\documentclass{amsart}
\usepackage{amsfonts,amssymb,amsmath,amsthm,mathrsfs}
\usepackage{url}
\usepackage{enumerate}

\newtheorem{theorem}{Theorem}[section]
\newtheorem{lemma}[theorem]{Lemma}

\theoremstyle{definition}

\newtheorem{remark}[theorem]{Remark}
\numberwithin{equation}{section}

\author[J. Lan]{Jiacheng Lan}
\address{Jiacheng Lan, College of Teacher Education, Lishui University, Lishui 323000, P. R. China}
\email{lanjc.ls@163.com}
\author[X. Tao]{Xiangxing Tao}
\address{Xiangxing Tao, Department of Mathematics, School of Science, Zhejiang University of Science and Technology,
Hangzhou 310023, P. R. China}
\email{xxtao@zust.edu.cn}
\author[G. Hu]{Guoen Hu}
\address{Guoen Hu, School of  Applied Mathematics, Beijing Normal University,
Zhuhai 519087,
P. R. China}
\email{guoenxx@163.com}
\thanks{ }

\begin{document}

\title[commutator]{weak type endpoint estimates for the commutators of rough singular integral operators}

\begin{abstract}
Let $\Omega$ be homogeneous of degree zero and have mean value zero on the unit sphere ${S}^{n-1}$, $T_{\Omega}$  be the convolution singular integral operator  with kernel $\frac{\Omega(x)}{|x|^n}$. For $b\in{\rm BMO}(\mathbb{R}^n)$, let $T_{\Omega,\,b}$ be the commutator of $T_{\Omega}$. In this paper, by establishing suitable sparse   dominations, the authors  establish some  weak type endpoint estimates of $L\log L$ type for $T_{\Omega,\,b}$ when $\Omega\in L^q(S^{n-1})$ for some $q\in (1,\,\infty]$.
\end{abstract}
\maketitle
\section{Introduction}
We will work on $\mathbb{R}^n$, $n\geq 2$. Let $\Omega$ be
homogeneous of degree zero, integrable and have mean value zero on
the unit sphere ${S}^{n-1}$. Define the  singular integral operator
${T}_{\Omega}$ by
\begin{eqnarray}\label{eq1.1}{T}_{\Omega}f(x)={\rm p.\,v.}\int_{\mathbb{R}^n} \frac
{\Omega( y')}{|y|^n}f(x-y)dy,\end{eqnarray}  where and in the following, $y'=y/|y|$ for $y\in\mathbb{R}^n$. This operator was introduced by
Calder\'on and Zygmund \cite{cz1}, and then studied by many authors
in the last sixty years.  Calder\'on and Zygmund \cite {cz2} proved that
if $\Omega\in L\log L({S}^{n-1})$, then $T_{\Omega}$ is bounded on
$L^p(\mathbb{R}^n)$ for $p\in (1,\,\infty)$.   Ricci
and Weiss \cite{rw} improved the result of Calder\'on-Zygmund, and
showed that $\Omega\in H^1(S^{n-1})$ guarantees the
$L^p(\mathbb{R}^n)$ boundedness on $L^p(\mathbb{R}^n)$ for $p\in
(1,\,\infty)$. Seeger \cite{se} showed that $\Omega\in L\log
L(S^{n-1})$ is a sufficient condition such that $T_{\Omega}$ is bounded
from $L^1(\mathbb{R}^n)$ to $L^{1,\,\infty}(\mathbb{R}^n)$.
For other works about the $L^p(\mathbb{R}^n)$ boundedness and weak type endpoint estimates for $T_{\Omega}$, we refer the papers
see \cite{chr2,duo,drf,fp,gs,rw,var} and the references therein.

Now let $T$ be a linear operator from $\mathcal{S}(\mathbb{R}^n)$ to $\mathcal{S}'(\mathbb{R}^n)
$ and $b\in{\rm BMO}(\mathbb{R}^n)$. The commutator of $T$ with symbol $b$, is defined by
$$T_bf(x)=b(x)Tf(x)-T(bf)(x).$$
A celebrated result of Coifman,  Rochberg and  Weiss \cite{crw} states  that if $T$ is a Calder\'on-Zygmund operator, then $T_{b}$ is bounded on $L^p(\mathbb{R}^n)$ for every $p\in (1,\,\infty)$ and also a converse result in
terms of the Riesz transforms. P\'erez \cite{perez1} considered the weak type endpoint estimate for the commutator of Calder\'on-Zygmund operator, and  proved the following result.
\begin{theorem}\label{thmperez}
Let $T$ be a Calder\'on-Zygmund operator and $b\in {\rm BMO}(\mathbb{R}^n)$. Then for any $\lambda>0$,
$$|\{x\in\mathbb{R}^n:\,|T_{b}f(x)|>\lambda\}|\lesssim_n\int_{\mathbb{R}^n}\frac{|f(x)|}{\lambda}\log \Big({\rm e}+\frac{|f(x)|}{\lambda}\Big)dx.$$
\end{theorem}
By Theorem \ref{thmperez}, we know that if $\Omega\in {\rm Lip}_{\alpha}(S^{n-1})$ with $\alpha\in (0,\,1]$, then for $b\in {\rm BMO}(\mathbb{R}^n)$,  $T_{\Omega,\,b}$,  the commutator of $T_{\Omega}$,  satisfies that,
\begin{eqnarray}\label{eq1.2}|\{x\in\mathbb{R}^n:\,|T_{\Omega,\,b}f(x)|>\lambda\}|\lesssim_n\int_{\mathbb{R}^n}\frac{|f(x)|}{\lambda}\log \Big({\rm e}+\frac{|f(x)|}{\lambda}\Big)dx.\end{eqnarray}

Let $p\in [1,\,\infty)$ and $w$ be a nonnegative, locally integrable function on $\mathbb{R}^n$. We say that   $w\in A_{p}(\mathbb{R}^n)$ if the $A_p$ constant $[w]_{A_p}$ is finite, with
$$[w]_{A_p}:=\sup_{Q}\Big(\frac{1}{|Q|}\int_Qw(x)dx\Big)\Big(\frac{1}{|Q|}\int_{Q}w^{1-p'}(x)dx\Big)^{p-1},\,\,\,p\in (1,\,\infty),$$
the  supremum is taken over all cubes in $\mathbb{R}^n$,  $p'=p/(p-1)$ and
$$[w]_{A_1}:=\sup_{x\in\mathbb{R}^n}\frac{Mw(x)}{w(x)},$$ see \cite{gra} for the properties of $A_p(\mathbb{R}^n)$.
For a weight $w\in A_{\infty}(\mathbb{R}^n)=\cup_{p\geq 1}A_p(\mathbb{R}^n)$, define $[w]_{A_{\infty}}$, the $A_{\infty}$ constant of $w$, by
$$[w]_{A_{\infty}}=\sup_{Q\subset \mathbb{R}^n}\frac{1}{w(Q)}\int_{Q}M(w\chi_Q)(x)dx,$$
see \cite{wil}. By the result of Duandikoetxea and Rubio de Francia \cite{drf}, and  the result in \cite{duo}, we know that if $\Omega\in L^q(S^{n-1})$ for some $q\in (1,\,\infty]$, then for $p\in (q',\,\infty)$ and $w\in A_{p/q'}(\mathbb{R}^n)$
$$\|T_{\Omega}f\|_{L^p(\mathbb{R}^n,\,w)}\lesssim_{n,p,w}\|f\|_{L^p(\mathbb{R}^n,\,w)}.$$
This, together with Theorem 2.13 in \cite{abkp}, tells us that if $\Omega\in L^q(S^{n-1})$ for $q\in (1\,\infty]$, then for $b\in {\rm BMO}(\mathbb{R}^n)$,
$$\|T_{\Omega,\,b}f\|_{L^p(\mathbb{R}^n,\,w)}\lesssim_{n,p,w}\|b\|_{{\rm BMO}(\mathbb{R}^n)}\|f\|_{L^p(\mathbb{R}^n,\,w)},\,\,p\in (q',\,\infty),\,\,w\in A_{p/q'}(\mathbb{R}^n).$$
However, as far as we know, there is no result concerning the weak type endpoint estimate for $T_{\Omega,\,b}$ when $\Omega$ only satisfies   size condition. In this paper, we consider this question. Our first result can be stated as follows.

\begin{theorem}\label{thm1.1}
Let $\Omega$ be homogeneous of degree zero and have mean value zero on $S^{n-1}$, $b\in {\rm BMO}(\mathbb{R}^n)$. Suppose that $\Omega\in L^{q}(S^{n-1})$  for some $q\in (1,\,\infty)$, then for any   $\lambda>0$ and weight $w$ such that $w^{q'}\in A_1(\mathbb{R}^n)$,
\begin{eqnarray*}&&w\big(\{x\in\mathbb{R}^n:\,|T_{\Omega,\,b}f(x)|>\lambda\}\big)\lesssim_{n,\,w}\int_{\mathbb{R}^n}\frac{D|f(x)|}{\lambda}\log \Big({\rm e}+\frac{D|f(x)|}{\lambda}\Big)w(x)dx,
\end{eqnarray*}
with $D=\|\Omega\|_{L^{q}(S^{n-1})}\|b\|_{{\rm BMO}(\mathbb{R}^n)}.$
\end{theorem}

In the last several years, considerable attention has been paid to the quantitative weighted bounds for $T_{\Omega}$ when $\Omega\in L^{\infty}(S^{n-1})$. The first result in this area was established by Hyt\"onen,  Roncal  and  Tapiola \cite{hrt}, who proved that for $p\in (1,\,\infty)$ and $w\in A_p(\mathbb{R}^n)$,
\begin{eqnarray}\label{eq1.3}\|T_{\Omega}f\|_{L^p(\mathbb{R}^n,\,w)}\lesssim_{n,p} \|\Omega\|_{L^{\infty}(S^{n-1})}[w]^{2\max\{1,\frac{1}{p-1}\}}\|f\|_{L^p(\mathbb{R}^n,\,w)}.\end{eqnarray}
Li,  P\'erez,  Rivera-Rios and  Roncal \cite{lpr} improved (\ref{eq1.3}) and  showed that for $p\in (1,\,\infty)$ and  $w\in A_p(\mathbb{R}^n)$
\begin{eqnarray}\label{eq1.4}&&\|T_{\Omega}f\|_{L^p(\mathbb{R}^n,\,w)}\lesssim_{n,p}[w]_{A_p}^{\frac{1}{p}} \big([w]_{A_{\infty}}^{\frac{1}{p'}}+[\sigma]_{A_{\infty}}^{\frac{1}{p}}\big)
\min\{[\sigma]_{A_{\infty}},\,[w]_{A_{\infty}}\}\|f\|_{L^p(\mathbb{R}^n,\,w)},\end{eqnarray}
where and in the following, for $w\in A_p(\mathbb{R}^n)$, $\sigma=w^{1-p'}$.
The estimate (\ref{eq1.4}),  via the method  in \cite{cpp}, implies the following quantitative weighted estimate
\begin{eqnarray*}\|T_{\Omega,\,b}f\|_{L^p(\mathbb{R}^n,\,w)}&\lesssim_{n,p} & [w]_{A_p}^{\frac{1}{p}} \big([w]_{A_{\infty}}^{\frac{1}{p'}}+[\sigma]_{A_{\infty}}^{\frac{1}{p}}\big)\min\{[\sigma]_{A_{\infty}},\,[w]_{A_{\infty}}\}\\
&&\quad\times ([w]_{A_{\infty}}+[\sigma]_{A_{\infty}})
\|f\|_{L^p(\mathbb{R}^n,\,w)}.\end{eqnarray*}
Rivera-R\'ios \cite{riv} established the sparse domination for $T_{\Omega,\,b}$ when $\Omega\in L^{\infty}(S^{n-1})$, and proved that for $p\in (1,\,\infty)$ and $w\in A_1(\mathbb{R}^n)$,
$$\|T_{\Omega,\,b}f\|_{L^p(\mathbb{R}^n,\,w)}\lesssim_{n,\,p} \|\Omega\|_{L^{\infty}(S^{n-1})} p'^3p^2[w]_{A_1}^{\frac{1}{p}}[w]_{A_{\infty}}^{1+\frac{1}{p'}}\|f\|_{L^p(\mathbb{R}^n,\,w)}.$$
Our second result is the following quantitative weighted weak type estimate for $T_{\Omega,\,b}$.
\begin{theorem}\label{thm1.2}
Let $\Omega$ be homogeneous of degree zero and have mean value zero on $S^{n-1}$, $b\in {\rm BMO}(\mathbb{R}^n)$. Suppose that $\Omega\in L^{\infty}(S^{n-1})$ and $w\in A_1(\mathbb{R}^n)$,  then for any $\lambda>0$,
\begin{eqnarray*}&&w(\{x\in\mathbb{R}^n:\,|T_{\Omega,\,b}f(x)|>\lambda\})\\
&&\quad\lesssim_n [w]_{A_1}[w]_{A_{\infty}}^2\log({\rm e}+[w]_{A_{\infty}})\int_{\mathbb{R}^n}\frac{D_{\infty}|f(x)|}{\lambda}\log \Big({\rm e}+\frac{D_{\infty}|f(x)|}{\lambda}\Big)w(x)dx,
\end{eqnarray*}
with $D_{\infty}=\|\Omega\|_{L^{\infty}(S^{n-1})}\|b\|_{{\rm BMO}(\mathbb{R}^n)}.$
\end{theorem}

\begin{remark}  Proofs of Theorem \ref{thm1.1} and Theorem \ref{thm1.2} depend essentially on the weak type endpoint estimates for the maximal operator defined by
\begin{eqnarray}\label{eq1.5}\mathscr{M}_{r,\,T_{\Omega}}f(x)=\sup_{Q\ni x}\Big(\frac{1}{|Q|}\int_Q|T_{\Omega}(f\chi_{\mathbb{R}^n\backslash 3Q})(\xi)|^rd\xi\Big)^{1/r},\end{eqnarray}
where the supremum is taken over all cubes $Q\subset \mathbb{R}^n$ containing $x$. This operator was introduced by Lerner \cite{ler4}, who proved that for any $r\in (1,\,\infty)$,
\begin{eqnarray}\label{eq1.6}\|M_{r,\,T_{\Omega}}f\|_{L^{1,\,\infty}(\mathbb{R}^n)}\lesssim r\|\Omega\|_{L^{\infty}(S^{n-1})}\|f\|_{L^1(\mathbb{R}^n)},
\end{eqnarray}
see \cite[Lemma 3.3]{ler4}. Although we can show that
$$\|M_{r,\,T_{\Omega}}f\|_{L^{1,\,\infty}(\mathbb{R}^n)}\lesssim_r\|\Omega\|_{L^{q}(S^{n-1})}\|f\|_{L^1(\mathbb{R}^n)},
$$
we do not know if there exists a $\alpha\in (0,\,\infty)$ such that the estimate $$\|M_{r,\,T_{\Omega}}f\|_{L^{1,\,\infty}(\mathbb{R}^n)}\lesssim r^{\alpha}\|\Omega\|_{L^{q}(S^{n-1})}\|f\|_{L^1(\mathbb{R}^n)}
$$
holds true when $\Omega\in L^q(S^{n-1)}$ for some $q\in (1,\,\infty)$. This is the main difficult which prevent us obtaining a desired quantitative weighted weak type endpoint estimates for $T_{\Omega,\,b}$ when $\Omega\in L^q(S^{n-1})$ for $q\in (1,\,\infty)$.
\end{remark}
In what follows, $C$ always denotes a
positive constant that is independent of the main parameters
involved but whose value may differ from line to line. We use the
symbol $A\lesssim B$ to denote that there exists a positive constant
$C$ such that $A\le CB$.  Specially, we use $A\lesssim_{n,p} B$ to denote that there exists a positive constant
$C$ depending only on $n,p$ such that $A\le CB$. Constant with subscript such as $c_1$,
does not change in different occurrences. For any set $E\subset\mathbb{R}^n$,
$\chi_E$ denotes its characteristic function.  For a cube
$Q\subset\mathbb{R}^n$ and $\lambda\in(0,\,\infty)$, we use
$\lambda Q$ to denote the cube with the same center as $Q$ and whose
side length is $\lambda$ times that of $Q$.  For a fixed cube $Q$, denote by $\mathcal{D}(Q)$ the set of dyadic cubes with respect to $Q$, that is, the cubes from $\mathcal{D}(Q)$ are formed by repeating subdivision of $Q$ and each of descendants into $2^n$ congruent subcubes. For a function $f$ and cube $Q$,  $\langle f\rangle_{Q}$ denotes the mean value of $f$ on $Q$, and $\langle |f|\rangle_{Q,\,r}=(\langle |f|^r\rangle_Q)^{1/r}$ for $r\in (0,\,\infty)$.

For a cube $Q$, $\beta\in (0,\,\infty)$ and suitable function $f$,
define $\|f\|_{L(\log L)^{\beta},\,Q}$ by
$$\|f\|_{L(\log L)^{\beta},\,Q}=\inf\Big\{\lambda>0:\,\frac{1}{|Q|}\int_{Q}\frac{|f(y)|}{\lambda}\log^{\beta}\Big({\rm e}+\frac{|f(y)|}{\lambda}\Big)dy\leq 1\Big\}.$$
Also, we define $\|h\|_{{\rm exp}L,\,Q}$ as
$$\|h\|_{{\rm exp}L,\,Q}=\inf\Big\{t>0:\,\frac{1}{|Q|}\int_{Q}{\rm exp}\Big(\frac{|h(y)|}{t}\Big){\rm d}y\leq 2\Big\}.$$
By the generalization of H\"older's inequality (see \cite[p. 64]{rr}), we know that for any cube $Q$ and suitable functions $f$ and $h$,
\begin{eqnarray}\label{eq1.final}
\int_{Q}|f(x)h(x)|dx\lesssim \|f\|_{L\log L,\,Q}\|h\|_{{\rm exp}L,\,Q}|Q|.
\end{eqnarray}

\section{Proof of Theorems}
Given an operator $T$, define the maximal operator ${M}_{\lambda,\,T}$ by
$${M}_{\lambda,\,T}f(x)=\sup_{Q\ni x}\Big(T(f\chi_{\mathbb{R}^n\backslash 3Q})\chi_{Q}\Big)^*(\lambda |Q|),\,\,(0<\lambda<1),$$
where the supremum is taken over all cubes $Q\subset \mathbb{R}^n$ containing $x$, and $h^*$ denotes the non-increasing rearrangement of $h$. This operator was introduced by Lerner \cite{ler4} and is useful in the study of weighted bounds for rough operators, see \cite{ler4, riv}.

\begin{lemma}\label{yinli2.1}
Let $\Omega$ be homogeneous of degree zero, have mean value zero and $\Omega\in L^{\infty}(S^{n-1})$. Then  for any  $\lambda\in (0,\,1)$,
$$\|{M}_{\lambda,\,T_{\Omega}}f\|_{L^{1,\,\infty}(\mathbb{R}^n)}\lesssim_n\|\Omega\|_{L^{\infty}(S^{n-1})}\big(1+\log
\big(\frac{1}{\lambda}\big)\big)\|f\|_{L^1(\mathbb{R}^n)}.
$$
\end{lemma}
Lemma \ref{yinli2.1} is Theorem 1.1 in \cite{ler4}.

For a function $\Omega$ on $S^{n-1}$, define $\|\Omega\|_{L\log L(S^{n-1})}^*$ by
$$\|\Omega\|^*_{L\log L(S^{n-1})}=\inf\Big\{\lambda>0:\,\int_{S^{n-1}}\frac{|\Omega(\theta)|}{\lambda}\log \Big({\rm e}+\frac{|\Omega(\theta)|}{\lambda}\Big)d\theta\leq 1\Big\}.$$
\begin{lemma}\label{2.seeger}
Let  $\Omega$ be homogeneous of degree zero, have mean value zero and $\|\Omega\|_{L\log L(S^{n-1})}^*<\infty$, then
$$\|T_{\Omega}f\|_{L^{1,\,\infty}(S^{n-1})}\lesssim \|\Omega\|^*_{L\log L(S^{n-1})}\|f\|_{L^1(\mathbb{R}^n)}.
$$
\end{lemma}
\begin{proof}   This lemma is essentially a corollary of estimate (3.1) in \cite{se}. At first, we claim that
\begin{eqnarray}\label{eq.orlicz}\int_{S^{n-1}}|\Omega(\theta)|\log ({\rm e}+\frac{|\Omega(\theta)|}{\|\Omega\|_{L^1(S^{n-1})}}\big)d\theta\lesssim \|\Omega\|_{L\log L(S^{n-1})}^*.
\end{eqnarray}
In fact, by homogeneity, it suffices to prove (\ref{eq.orlicz}) for the case  $\|\Omega\|_{L^1(S^{n-1})}=1$.  Let $$\lambda_0=\int_{S^{n-1}}|\Omega(\theta)|\log ({\rm e}+|\Omega(\theta)|\big)d\theta.$$  We consider the following two cases.

{\it Case I}.  $\lambda_0>{\rm e}^{10}.$ Let $S_0=\{\theta\in S^{n-1}:\, |\Omega(\theta)|\leq 2\},$ and
$$S_k=\{\theta\in S^{n-1}:\,2^{k}<|\Omega(\theta)|\leq 2^{k+1}\},\,\, k\in\mathbb{N}.$$
Set $k_0\in\mathbb{N}$such that $2^{k_0-1}<\lambda_0\leq 2^{k_0}.$ Then $k_0\leq \lambda_0/8$
\begin{eqnarray*}
\int_{S^{n-1}}\frac{|\Omega(\theta)|}{\lambda_0}\log\big
({\rm e}+\frac{|\Omega(\theta)|}{\lambda_0}\Big )d\theta&>&\lambda_0^{-1}\sum_{k=k_0+1}^{\infty}|S_k|2^{k}(k-k_0)+ \lambda_0^{-1}\sum_{k\leq k_0}|S_k|2^{k}\\
&>&\lambda^{-1}_0 \Big (\sum_{k=1}^{\infty}2^{
k}k|S_k|+|S_0|\Big )\\
&&-\lambda^{-1}_0\Big (k_0\sum_{k\geq k_0+1}2^{
k}|S_k|+\sum_{1\leq k\leq k_0}k2^{ k}|S_{k}|\Big ).
\end{eqnarray*}
Obviously,
$$\sum_{k=1}^{\infty}2^{
k}k|S_k|+|S_0|\geq \frac{1}{4}\int_{S^{n-1}}|\Omega(\theta)|\log ({\rm e}+|\Omega(\theta)|)d\theta=\frac{\lambda_0}{4},
$$
and
\begin{eqnarray*}
k_0\sum_{k\geq k_0+1}2^{ k}|S_k|+\sum_{1\leq k\leq
k_0}k2^{ k}|S_{k}|\leq k_0\sum_{k\geq 1}2^{  k}|S_k| \leq k_0\|\Omega\|_{L^1(S^{n-1})}.
\end{eqnarray*}
Recall that $\|\Omega\|_{L^1(S^{n-1})}=1$. It then follows that
$$\int_{S^{n-1}}\frac{|\Omega(\theta)|}{\lambda_0}\log\big
({\rm e}+\frac{|\Omega(\theta)|}{\lambda_0}\Big )d\theta>\frac{1}{8}.
$$
This in turn leads to that
$$\|\Omega\|_{L\log L(S^{n-1})}^*>\lambda_0/8.
$$

{\it Case II}. $\lambda_0\leq {\rm e}^{10}.$ Let $\lambda>0$ satisfies that
\begin{eqnarray}\label{eq.contra}\int_{S^{n-1}}\frac{|\Omega(\theta)|}{\lambda}\log\Big ({\rm e}+\frac{|\Omega(\theta)|}{\lambda}\Big )d\theta\leq 1.\end{eqnarray}
If $10{\rm e}^{10}\lambda<\lambda_0$, we then have that
\begin{eqnarray*}
&&\int_{S^{n-1}}\frac{|\Omega(\theta)|}{\lambda_0}\log\Big
({\rm e}+\frac{|\Omega(\theta)|}{\lambda_0}\Big )d\theta \leq\int_{Q}\frac{|\Omega(\theta)|}{10{\rm e}^{10}\lambda}\log\Big
({\rm e}+\frac{|\Omega(\theta)|}{10{\rm e}^{10}\lambda}\Big )d\theta\leq (10{\rm e}^{10})^{-1}.
\end{eqnarray*}
On the other hand, a trivial computation gives us that
\begin{eqnarray*}
\int_{S^{n-1}}\frac{|\Omega(\theta)|}{\lambda_0}\log\Big
({\rm e}+\frac{|\Omega(\theta)|}{\lambda_0}\Big )d\theta
&>&\int_{S^{n-1}}\frac{|\Omega(\theta)|}{{\rm e}^{10}}\log \Big
({\rm e}+\frac{|\Omega(\theta)|}{{\rm e}^{10}}\Big )d\theta\\
&>&\int_{S^{n-1}}|\Omega(\theta)|\log({\rm e}+|\Omega(\theta)|)d\theta(10{\rm e}^{10})^{-1}\\
&>&(10{\rm e}^{10})^{-1}.
\end{eqnarray*}
This is a contradiction. Thus, the positive numbers $\lambda$ in (\ref{eq.contra}) satisfy $\lambda\geq (10{\rm e}^{10})^{-1}\lambda_0.$  Inequality (\ref{eq.orlicz}) holds true in this case.

We now conclude the proof of Lemma \ref{2.seeger}. By the result of Seeger (see inequality (3.1) in \cite{se}), we know that  if $\Omega\in L\log L(S^{n-1})$, then  \begin{eqnarray*}\|T_{\Omega}f\|_{L^{1,\,\infty}(\mathbb{R}^n)}&\lesssim_n & \Big[\|T_{\Omega}\|_{L^2(\mathbb{R}^n)\rightarrow L^2(\mathbb{R}^n)}+\|\Omega\|_{L^1(S^{n-1})}\\
&&+\int_{S^{n-1}}|\Omega(\theta)|\Big(1+\log^+\big(|\Omega(\theta)|/\|\Omega\|_{L^1(S^{n-1})}\big)\Big)d\theta\Big]  \|f\|_{L^1(\mathbb{R}^n)},
\end{eqnarray*}
where $\log^+ s=\log s$ if $s>1$ and $\log ^+s=0$ if $s\in (0,\,1]$. Thus by (\ref{eq.orlicz}),
\begin{eqnarray*}\|T_{\Omega}f\|_{L^{1,\,\infty}(\mathbb{R}^n)}\lesssim_n  \big[\|T_{\Omega}\|_{L^2(\mathbb{R}^n)\rightarrow L^2(\mathbb{R}^n)}+\|\Omega\|_{L^1(S^{n-1})}+\|\Omega\|^*_{L\log L(S^{n-1})}\big]  \|f\|_{L^1(\mathbb{R}^n)}.
\end{eqnarray*}
On the other hand, we know that
$$\|T_{\Omega}f\|_{L^2(\mathbb{R}^n)}\lesssim \big[1+\|\Omega\|_{L\log L(S^{n-1})}\big]\|f\|_{L^2(\mathbb{R}^n)},$$with $$\|\Omega\|_{L\log L(S^{n-1})}=\int_{S^{n-1}}|\Omega(\theta)|\big(1+\log^+|\Omega(\theta)|\big)d\theta.$$
see \cite[Theorem 4.2.10]{gra2}.
The last two inequality,  along with homogeneity, yields
\begin{eqnarray*}\|T_{\Omega}f\|_{L^{1,\,\infty}(\mathbb{R}^n)}\lesssim_n  \|\Omega\|_{L\log L(S^{n-1})}^* \|f\|_{L^1(\mathbb{R}^n)},
\end{eqnarray*}
and  completes the proof of Lemma \ref{2.seeger}.\qed
\end{proof}
\begin{lemma}\label{yinli2.2}
Let $\Omega$ be homogeneous of degree zero, have mean value zero and $\Omega\in L^{q}(S^{n-1})$ for some $q\in (1,\,\infty)$. Then  for any  $\lambda\in (0,\,1)$ and $\varepsilon\in (0,\,\min\{1,\,q-1\})$,
$$\|{M}_{\lambda,\,T_{\Omega}}f\|_{L^{1,\,\infty}(\mathbb{R}^n)}\lesssim_{q,\,\varepsilon}\|\Omega\|_{L^{q}(S^{n-1})}
\big(\frac{1}{\lambda}\big)^{\frac{1+2\varepsilon}{q}}
\|f\|_{L^1(\mathbb{R}^n)}.
$$

\end{lemma}
\begin{proof}
For $\lambda\in (0,\,1)$, let $M_{0,\,\lambda}$ be the operator
$$M_{0,\,\lambda}h(x)=\sup_{Q\ni x}(h\chi_{Q})^*(\lambda|Q|),$$ see \cite{john,stro}.
It is well known that for  $\alpha>0$,
$$|\{x\in\mathbb{R}^n:\,M_{0,\,\lambda}f(x)>\alpha\}|\lesssim \lambda^{-1}|\{x\in\mathbb{R}^n:\,|f(x)|>\alpha\}|.$$
Let $S$ be a linear operator which is bounded from $L^1(\mathbb{R}^n)$ to $L^{1,\,\infty}(\mathbb{R}^n)$ with bound $1$. We claim that the operator $S_{\lambda}^{\star}$ defined by
$$S_{\lambda}^{\star}f(x)=\sup_{Q\ni x}\big(S(f\chi_Q)\big)^*(\lambda |Q|)
$$
is bounded from $L^1(\mathbb{R}^n)$ to $L^{1,\,\infty}(\mathbb{R}^n)$ with bound  $C_n\lambda^{-1}$.  To prove this, let
$$E_{\alpha}=\{x\in\mathbb{R}^n:\,S_{\lambda}^{\star}f(x)>\alpha\}.
$$
For each $x\in E_{\alpha}$, we can choose a cube $Q$ such that $Q\ni x$ and
$$|\{y\in Q:\,|S(f\chi_{Q})(y)|>\alpha\}|>\lambda |Q|.$$
This, via the weak type $(1,\,1)$ boundedness of $S$, tells us that
$$|Q|\le  \frac{1}{\alpha\lambda}\int_{Q}|f(y)|dy,$$
and so $Mf(x)\geq \alpha\lambda$. Therefore,
$$|E_{\alpha}|\leq |\{x\in\mathbb{R}^n:\,Mf(x)>\lambda\alpha\}|\lesssim\frac{1}{\lambda\alpha}\|f\|_{L^1(\mathbb{R}^n)}.$$
This verifies our claim.

We now conclude the proof of Lemma \ref{yinli2.2}.  Using the estimate $\log t\leq t^{\varepsilon}/{\varepsilon}$ when $t>1$ and $\varepsilon>0$, we can verify by homogeneity that
$$\|\Omega\|_{L\log L(S^{n-1})}^*\lesssim _{\varepsilon}\|\Omega\|_{L^{1+\varepsilon}(S^{n-1})}.$$
This,  along with Lemma \ref{2.seeger}, tells us that for $\varepsilon>0$,
$$\|T_{\Omega}f\|_{L^{1,\,\infty}(\mathbb{R}^n)}\lesssim_{n,\varepsilon} \|\Omega\|_{L^{1+\varepsilon}(S^{n-1})}\|f\|_{L^1(\mathbb{R}^n)}.
$$Observe that
$$M_{\lambda,\,T_{\Omega}}f(x)\leq M_{0,\,\frac{\lambda}{2}}T_{\Omega}f(x)+\sup_{Q\ni x}\big(T_{\Omega}(f\chi_{3Q})\chi_{Q}\big)^*(\frac{\lambda}{2}|Q|),$$
and
$$\sup_{Q\ni x}\big(T_{\Omega}(f\chi_{3Q})\chi_{Q}\big)^*(\frac{\lambda}{2}|Q|)\leq \sup_{Q\ni x}\big(T_{\Omega}(f\chi_{Q})\chi_{Q}\big)^*(\frac{1}{3^{n}}\frac{\lambda}{2}|Q|).$$
Our claim states that
\begin{eqnarray}\label{gongshi2.1}\|M_{\lambda,T_{\Omega}}f\|_{L^{1,\,\infty}(\mathbb{R}^n)}\lesssim_{\varepsilon} \frac{1}{\lambda}\|\Omega\|_{L^{1+\varepsilon}(S^{n-1})}\|f\|_{L^1(\mathbb{R}^n)}.
\end{eqnarray}
Now let $\Omega\in L^q(S^{n-1})$, have mean value zero on $S^{n-1}$. Without loss of generality, we assume that $\|\Omega\|_{L^q(S^{n-1})}=1.$
Set
$$t_0=\big(\frac{1}{\lambda}\big)^{\frac{1+\varepsilon}{q}}\big[1+\log \big(\frac{1}{\lambda}\big)\big]^{-\frac{1+\varepsilon}{q}}.$$
Let
$$\Omega^{t_0}(\theta)=\Omega(\theta)\chi_{\{|\Omega(\theta)|>t_0\}}(\theta),\,\,\,
\Omega_{t_0}(\theta)=\Omega(\theta)\chi_{\{|\Omega(\theta)|\leq t_0\}}(\theta),$$
and
$$\widetilde{\Omega}^{t_0}(\theta)=\Omega^{t_0}(\theta)-A^{t_0},\,\,\,
\widetilde{\Omega}_{t_0}(\theta)=\Omega_{t_0}(\theta)-A_{t_0},$$
where
$$A^{t_0}=\frac{1}{|S^{n-1}|}\int_{S^{n-1}}\Omega^{t_0}(\theta)d\theta,\,\,A_{t_0}=
\frac{1}{|S^{n-1}|}\int_{S^{n-1}}\Omega_{t_0}(\theta)d\theta.
$$
Both of $\widetilde{\Omega}^{t_0}$ and $\widetilde{\Omega}_{t_0}$ have mean value zero. Moreover,
$$\|\widetilde{\Omega}^{t_0}\|_{L^{1+\varepsilon}(S^{n-1})}\lesssim t_0^{1-\frac{q}{1+\varepsilon}},\,\,
\|\widetilde{\Omega}_{t_0}\|_{L^{\infty}(S^{n-1})}\lesssim t_0,
$$
and $\Omega(\theta)=\widetilde{\Omega}^{t_0}(\theta)+\widetilde{\Omega}_{t_0}(\theta).$
Applying Lemma \ref{yinli2.1} and (\ref{gongshi2.1}), we deduce that
\begin{eqnarray*}\|M_{\lambda,T_{\Omega}}f\|_{L^{1,\,\infty}(\mathbb{R}^n)}&\lesssim & \|M_{\lambda,T_{\widetilde{\Omega}^{t_0}}}f\|_{L^{1,\,\infty}(\mathbb{R}^n)}+ \|M_{\lambda,T_{\widetilde{\Omega}_{t_0}}}f\|_{L^{1,\,\infty}(\mathbb{R}^n)}\\ &\lesssim_{\varepsilon} &\frac{1}{\lambda}\|\widetilde{\Omega}^{t_0}\|_{L^{1+\varepsilon}(S^{n-1})}\|f\|_{L^1(\mathbb{R}^n)}\\
&&+\big[1+\log (\frac{1}{\lambda})\big]\|\widetilde{\Omega}_{t_0}\|_{L^{\infty}(S^{n-1})}\|f\|_{L^1(\mathbb{R}^n)}\\
&\lesssim_{q,\,\varepsilon}&
\big(\frac{1}{\lambda}\big)^{\frac{1+\varepsilon}{q}}\big[1+\log (\frac{1}{\lambda}\big)\big]^{1-\frac{1+\varepsilon}{q}}
\|f\|_{L^1(\mathbb{R}^n)}\\
&\lesssim_{q,\,\varepsilon}&
\big(\frac{1}{\lambda}\big)^{\frac{1+2\varepsilon}{q}}
\|f\|_{L^1(\mathbb{R}^n)}.
\end{eqnarray*}
where in the last inequality, we again invoked the fact that
$\log t\leq t^{\alpha}/\alpha$ for all $t>1$ and $\alpha>0$.
This completes the proof of Lemma \ref{yinli2.2}.\qed
\end{proof}
\begin{lemma}\label{lem2.3}
Let $r \in (1,\,\infty)$ and $w$ be a weight. The following two statements are equivalent.
\begin{itemize}
\item[\rm (i)] $w\in A_1(\mathbb{R}^n)$ and  $w^{1-p'}\in A_{p'/r}(\mathbb{R}^n)$ for some $p\in (1,\,r')$;
\item[\rm (ii)] $w^r\in A_1(\mathbb{R}^n)$.
\end{itemize}
\end{lemma}
\begin{proof} Let  $w\in A_{1}(\mathbb{R}^n)$ and $w^{1-p'}\in A_{p'/r}(\mathbb{R}^n)$ for some $p\in (1,\,r')$, then for any cube $Q\subset \mathbb{R}^n$,
\begin{eqnarray*}
\Big(\frac{1}{|Q|}\int_Qw^{1-p'}(x)dx\Big)\Big(\frac{1}{|Q|}\int_{Q}w^{r\frac{p'-1}{p'-r}}(x)dx\Big)^{\frac{p'}{r}-1}\le  [w^{1-p'}]_{A_{p'/r}},
\end{eqnarray*}
and so
\begin{eqnarray*}
\frac{1}{|Q|}\int_{Q}w^{r\frac{p'-1}{p'-r}}(x)dx&\le & [w^{1-p'}]_{A_{p'/r}}^{\frac{1}{\frac{p'}{r}-1}}\Big(\frac{1}{|Q|}\int_Qw^{1-p'}(x)dx\Big)^{-\frac{1}{\frac{p'}{r}-1}}\\
&\le&[w^{1-p'}]_{A_{p'/r}}^{\frac{1}{\frac{p'}{r}-1}}[w]_{A_1}^{\frac{1}{\frac{p'}{r}-1}\frac{1}{p-1}}\Big(\frac{1}{|Q|}\int_Qw(x)dx\Big)^{\frac{1}{\frac{p'}{r}-1}\frac{1}{p-1}}\\
&\le&[w^{1-p'}]_{A_{p'/r}}^{\frac{1}{\frac{p'}{r}-1}}[w]_{A_1}^{\frac{1}{\frac{p'}{r}-1}\frac{1}{p-1}}\big({\rm essinf}_{y\in Q}w(y)\Big) ^{\frac{p'-1}{\frac{p'}{r}-1}},
\end{eqnarray*}
where the second inequality follows from the fact that
$$\Big(\frac{1}{|Q|}\int_Qw(x)dx\Big)\Big(\frac{1}{|Q|}\int_Qw^{1-p'}(x)dx\Big)^{p-1}\geq 1.
$$
We thus deduce that $w^r\in A_1(\mathbb{R}^n)$, with $[w^r]_{A_1}\leq [w^{1-p'}]_{A_{p'/r}}^{\frac{1}{p'-1}}[w]_{A_1}^{r}$.

Let $w^r\in A_1(\mathbb{R}^n)$. By the reverse H\"older inequality, we know that
$w^{r\frac{p'-1}{p'-r}}\in A_1(\mathbb{R}^n)$ for some $p\in (1,\,r')$, and $[w]_{A_1}\leq [w^r]_{A_1}$, $[w^{r\frac{p'-1}{p'-r}}]_{A_1}\le [w^r]_{A_1}^{(p'-1)/(p'-r)}$. Thus for any cube $Q\subset \mathbb{R}^n$,
\begin{eqnarray*}
&&\Big(\frac{1}{|Q|}\int_Qw^{1-p'}(x)dx\Big)\Big(\frac{1}{|Q|}\int_{Q}w^{r\frac{p'-1}{p'-r}}(x)dx\Big)^{\frac{p'}{r}-1}\\
&&\quad\le \big[{\rm essinf}_{y\in Q}w(y)\big]^{1-p'}[w^{r\frac{p'-1}{p'-r}}]_{A_1}^{\frac{p'}{r}-1}
\big[{\rm essinf}_{y\in Q}w(y)\big]^{p'-1}\le [w^r]_{A_1}^{\frac{p'-1}{r}}.
\end{eqnarray*}
This shows that $w^{1-p'}\in A_{p'/r}(\mathbb{R}^n)$.\qed
\end{proof}

\begin{lemma}\label{lem2.4}
Let $T$ be a sublinear operator. Suppose that there exists a  constant $\tau\in (0,\,1)$, such that for all $\lambda\in (0,\,1/2)$,
$$\|M_{\lambda,\,T}f\|_{L^{1,\,\infty}(\mathbb{R}^n)}\le  \lambda^{-\tau}\|f\|_{L^1(\mathbb{R}^n)}.$$
Then for  $p_0\in (1,\,1/\tau)$,
$$\|\mathscr{M}_{p_0,\,T}f\|_{L^{1,\,\infty}(\mathbb{R}^n)}\le 2^{2+\frac{4}{1-\tau p_0}}\|f\|_{L^1(\mathbb{R}^n)}
,$$
where $\mathscr{M}_{p_0,\,T}$ is the maximal operator defined as (\ref{eq1.5}).
\end{lemma}\begin{proof} We employ the argument used in the proof of Lemma 3.3 in \cite{ler4}.  As it was proved in \cite{ler4}, $$\mathscr{M}_{p_0,\,T}f(x)\leq\Big(\int_{0}^1\big(M_{\lambda,\,T}f(x)\big)^{p_0}d\lambda\Big)^{\frac{1}{p_0}}.
$$For $N>0$, denote
$$G_{p_0,\,T,\,N}f(x)=\Big(\int_{0}^1\big(\min\{M_{\lambda,\,T}f(x),\,N\}\big)^{p_0}d\lambda\Big)^{\frac{1}{p_0}},
$$
and
$$\mu_f(\alpha,\,R)=|\{x\in\mathbb{R}^n:\,|x|\leq R,\, |f(x)|>\alpha\}|,\,\,\,\alpha,\,R>0.$$
Let $p_0\in (1,\,\infty)$ such that $\tau p_0\in (0,\,1)$, $k=\lfloor\frac{4}{1-\tau p_0}\rfloor+1$, where and in the following, for $a\in\mathbb{R}$, $\lfloor a\rfloor$ denotes the integer part of $a$. By H\"older's inequality,
\begin{eqnarray*}
G_{p_0,\,T,\,N}f(x)&\leq &\Big(\int^{\frac{1}{2^{kp_0}}}\big(\min\{M_{\lambda,\,T}f(x),\,N\}\big)^{p_0}d\lambda\Big)^{\frac{1}{p_0}}+M_{1/2^{kp_0},\,T}f(x)\\
&\leq&\frac{1}{2^{k-1}}G_{kp_0,\,T,\,N}f(x)+M_{1/2^{kp_0},\,T}f(x).
\end{eqnarray*}
Therefore,
\begin{eqnarray*}
\mu_{G_{p_0,\,T,\,N}f}(\alpha,\,R)&\leq &\mu_{G_{kp_0,\,T,\,N}f}(2^{k-2}\alpha,\,R)+\mu_{M_{1/2^{kp_0},\,T}f}(\alpha/2,\,R)\\
&\leq&\mu_{G_{kp_0,\,T,\,N}f}(2^{k-2}\alpha,\,R)+\frac{1}{\alpha}2^{\tau kp_0+1}\|f\|_{L^1(\mathbb{R}^n)}.
\end{eqnarray*}
Repeating the last inequality $j$ times, we have that
\begin{eqnarray*}
\mu_{G_{p_0,\,T,\,N}f}(\alpha,\,R)&\leq &\mu_{G_{k^jp_0,\,T,\,N}f}(2^{j(k-2)}\alpha,\,R)\\
&&+\frac{2^{k-2}}{\alpha}\sum_{l=1}^j\Big(\frac{2^{\tau k p_0+1}}{2^{k-2}}\Big)^l\|f\|_{L^1(\mathbb{R}^n)}.
\end{eqnarray*}
Since $G_{p_0,\,T,\,N}f$ is uniformly bounded in $p_0$, we obtain that $\mu_{G_{k^{j}p_0,\,T,\,N}f}(\alpha,\,R)\rightarrow 0$ as $j\rightarrow\infty$.
We finally deduce that
$$ \mu_{G_{p_0,\,T,\,N}f}(\alpha,\,R)\le  2^{2+\frac{4}{1-\tau p_0}} \frac{1}{\alpha}\|f\|_{L^1(\mathbb{R}^n)}.$$
This completes the proof of Lemma \ref{lem2.4}.\qed\end{proof}
Let $\eta\in (0,\,1)$ and $\mathcal{S}=\{Q_j\}$ be a family of cubes. We say that $\mathcal{S}$ is $\eta$-sparse,  if for each fixed $Q\in \mathcal{S}$, there exists a measurable subset $E_Q\subset Q$, such that $|E_Q|\geq \eta|Q|$ and $E_{Q}$'s are pairwise disjoint.
For  sparse family $\mathcal{S}$ and constants $\beta$, $r\in[0,\,\infty)$, we define the bilinear sparse operator $\mathcal{A}_{\mathcal{S};\,L(\log L)^{\beta},\,L^{r}}$  by
$$\mathcal{A}_{\mathcal{S};\,L(\log L)^{\beta},L^r}(f,g)=\sum_{Q\in\mathcal{S}}|Q|\|f\|_{L(\log L)^{\beta},\,Q}\langle|g|\rangle_{Q,\,r}.$$
We denote $\mathcal{A}_{\mathcal{S};\,L(\log L)^1,\,L^r}$ by $\mathcal{A}_{\mathcal{S};\,L\log L,\,L^r}$ for simplicity,
and   $\mathcal{A}_{\mathcal{S};\,L(\log L)^{0},L^r}$ by $\mathcal{A}_{\mathcal{S};\,L,\,L^r}$.

\begin{lemma}\label{lem2.5}Let  $\alpha,\,\beta\in \mathbb{N}\cup\{0\}$ and $U$ be an operator. Suppose that for any  $r\in (1,\,3/2)$, and bounded function $f$ with compact support, there exists a sparse family of cubes $\mathcal{S}$, such that for any function $g\in L^1(\mathbb{R}^n)$,
\begin{eqnarray}\label{eq4.1}\Big|\int_{\mathbb{R}^n}Uf(x)g(x)dx\Big|\leq r'^{\alpha}\mathcal{A}_{\mathcal{S};\,L(\log L)^{\beta},\,L^r}(f,\,g).\end{eqnarray}
Then for any $u\in A_1(\mathbb{R}^n)$ and bounded function $f$ with compact support,
\begin{eqnarray*}&&w(\{x\in\mathbb{R}^n:\, |Uf(x)|>\lambda\})\\
&&\quad\lesssim_{n,\,\alpha,\,\beta} [w]_{A_{\infty}}^\alpha\log^{1+\beta}({\rm e}+[w]_{A_{\infty}})[w]_{A_1}\int_{\mathbb{R}^d}\frac{|f(x)|}{\lambda}\log ^{\beta}\Big({\rm e}+\frac{|f(x)|}{\lambda}\Big)w(x)dx.\end{eqnarray*}
\end{lemma}
Lemma \ref{lem2.5} is Corollary 3.6 in \cite{hulai}.

\begin{theorem}\label{thm3.1}Let $p_0\in (1,\,\infty)$, $r\in (1,\,\infty)$, $b\in {\rm BMO}(\mathbb{R}^n)$, $T$ be a linear operator and $T_b$ be the commutator of $T$. Suppose that both of operators $T$ and $\mathscr{M}_{p_0,\,T}$   are bounded from $L^1(\mathbb{R}^n)$ to $L^{1,\,\infty}(\mathbb{R}^n)$ with bound $1$. Then for
bounded functions $f$ with compact supports, there exists a $\frac{1}{2}\frac{1}{3^n}$-sparse family $\mathcal{S}$  and functions ${\rm H}_1f$, ${\rm H}_2f$, such that for each   function $g\in L^{rp_0'}_{{\rm loc}}(\mathbb{R}^n)$,
\begin{eqnarray}\label{eq4.2x}\Big|\int_{\mathbb{R}^n}{\rm H}_1f(x)g(x)dx\Big|\lesssim_{n}\|b\|_{{\rm BMO}(\mathbb{R}^n)} r'p_0'\mathcal{A}_{\mathcal{S};\,L^1,\,L^{rp_0'}}(f,\,g),\end{eqnarray}
\begin{eqnarray}\label{eq4.3x}\Big|\int_{\mathbb{R}^n}{\rm H}_2f(x)g(x)dx\Big|\lesssim_{n}\|b\|_{{\rm BMO}(\mathbb{R}^n)}\mathcal{A}_{\mathcal{S};\,L\log L,\,L^{p_0'}}(f,\,g),\end{eqnarray}
and for a. e. $x\in\mathbb{R}^n$,
$$T_{b}f(x)={\rm H}_1f(x)+{\rm H}_2f(x).
$$
\end{theorem}
\begin{proof}
We will employ the ideas in \cite{ler4}, see also the proof of Theorem 3.2 in \cite{hulai}. Without loss of generality, we may assume that $\|b\|_{{\rm BMO}(\mathbb{}R^n)}=1$. For a fixed cube $Q_0$, define the  local analogy of $\mathscr{M}_{p_0,\,T}$ by
$$ \mathscr{M}_{p_0,\,T;\,Q_0}f(x)=\sup_{Q\ni x,\, Q\subset Q_0}\Big(\frac{1}{|Q|}\int_{Q}|T(f\chi_{3Q_0\backslash 3Q})(y)|^{p_0}dy\Big)^{\frac{1}{p_0}}.$$
Let $E=\cup_{j=1}^4E_j$ with
$$E_1=\big\{x\in Q_0:\, |T(f\chi_{3Q_0})(x)|>D\langle|f|\rangle_{3Q_0}\big\},$$
$$E_2=\big\{x\in Q_0:\, |T\big((b-\langle b\rangle_{Q_0})f\chi_{3Q_0}\big)(x)|>D\langle|(b-\langle b\rangle_{Q_0})f|\rangle_{3Q_0}\big\},$$
$$E_3=\{x\in Q_0:\,\mathscr{M}_{p_0,\,T;\,Q_0}f(x)>D\langle |f|\rangle_{3Q_0}\},$$ and
$$E_4=\big\{x\in Q_0:\, \mathscr{M}_{p_0,\,T_{\Omega};\,Q_0}\big((b-\langle b\rangle_{Q_0})f\big)(x)>D\langle|b-\langle b\rangle_{Q_0}||f|\rangle_{Q_0}\big\},$$
where $D$  is a positive constant. If we choose $D$ large enough, it then follows from the weak type $(1,\,1)$ boundedness of $T$ and $\mathscr{M}_{p_0,\,T}$ that
$$|E|\le \frac{1}{2^{n+2}}|Q_0|.$$
Now on the cube $Q_0$, we apply the Calder\'on-Zygmund decomposition to $\chi_{E}$ at level $\frac{1}{2^{n+1}}$, and obtain pairwise disjoint cubes $\{P_j\}\subset \mathcal{D}(Q_0)$, such that
$$\frac{1}{2^{n+1}}|P_j|\leq |P_j\cap E|\leq \frac{1}{2}|P_j|$$
and $|E\backslash\cup_jP_j|=0$.  Observe that $\sum_j|P_j|\leq \frac{1}{2}|Q_0|$. Let
$$G_{Q_0}^1(x)=(b(x)-\langle b\rangle_{Q_0})T(f\chi_{3Q_0})\chi_{Q_0\backslash \cup_{l}P_l}(x)+
\sum_{l}(b(x)-\langle b\rangle_{Q_0})T(f\chi_{3Q_0\backslash 3P_l})\chi_{P_l}(x),
$$
$$G_{Q_0}^2(x)=T\big((b-\langle b\rangle_{Q_0})f\chi_{3Q_0}\big)\chi_{Q_0\backslash \cup_{l}P_l}(x)+
\sum_{l}T\big((b-\langle b\rangle_{Q_0})f\chi_{3Q_0\backslash 3P_l}\big)\chi_{P_l}(x).
$$It then follows that
$$T_{b}(f\chi_{3Q_0})(x)\chi_{Q_0}(x)=G_{Q_0}^1(x)+G_{Q_0}^2(x)+\sum_{l}T_{b}(f\chi_{3P_l})(x)\chi_{P_l}(x).
$$

We now estimate $G_{Q_0}^1$ and $G_{Q_0}^2$. By (\ref{eq1.final}) and the John-Nirenberg inequality (see \cite[p.128]{gra}), we know that
\begin{eqnarray*}\int_{Q_0}|b(x)-\langle b\rangle_{Q_0}||h(x)|dx&\lesssim& |Q_0|\|b-\langle b\rangle_{Q_0}\|_{{\rm exp}L,\,Q}\|h\|_{L\log L,\,Q_0}\\
&\lesssim& |Q_0|\|b\|_{{\rm BMO}(\mathbb{R}^n)}\|h\|_{L\log L,\,Q_0}.\end{eqnarray*}
This, along with  the fact that $|E\backslash\cup_jP_j|=0$,  implies that
$$\Big|\int_{Q_0\backslash \cup_{l}P_l}(b(x)-\langle b\rangle_{Q_0})T(f\chi_{3Q_0})(x)g(x)dx\Big|\lesssim \langle |f|\rangle_{3Q_0}\|g\|_{L\log L,\,Q_0}|Q_0|,$$
and
$$\Big|\int_{Q_0\backslash \cup_{l}P_l}T\big((b-\langle b\rangle_{Q_0})f\chi_{3Q_0}\big)(x)g(x)dx\Big|\lesssim \langle|f|\rangle_{L\log L,\,3Q_0}\langle|g|\rangle_{Q_0}|Q_0|.$$
On the other hand, the fact that $P_j\cap E^c\not =\emptyset$ tells us that
\begin{eqnarray*}
&&\sum_{l}\Big|\int_{P_l}(b(x)-\langle b\rangle_{Q_0})T(f\chi_{3Q_0\backslash 3P_l})(x)g(x)dx\Big|\\
&&\quad\lesssim \sum_l\Big(\int_{P_l}|b(x)-\langle b\rangle_{Q_0}|^{p_0'}|g(x)|^{p_0'}dx\Big)^{\frac{1}{p_0'}}
\Big(\int_{P_l}|T(f\chi_{3Q_0\backslash 3P_l})(x)|^{p_0}dx\Big)^{p_0}\\
&&\quad\lesssim \sum_l\Big(\int_{P_l}|b(x)-\langle b\rangle_{Q_0}|^{p_0' r '}\Big)^{\frac{1}{p_0' r '}}
|P_l|^{\frac{1}{p_0'r}+\frac{1}{p_0}}\langle |g|\rangle_{P_l,\,p_0'r}\inf_{y\in P_l}\mathscr{M}_{T,\,p_0,Q_0}f(y)\\
&&\quad\lesssim r'p_0'\langle |f|\rangle_{3Q_0}\sum_l|P_l|\langle |g|\rangle_{P_l,\,rp_0'}\lesssim r'p_0'\langle |f|\rangle_{3Q_0}\langle |g|\rangle_{Q_0,\,rp_0'}|Q_0|,
\end{eqnarray*}here we have invoked the following estimate
$$\Big(\int_{Q_0}|b(x)-\langle b\rangle_{Q_0}|^{p_0' r '}dx\Big)^{\frac{1}{p_0' r '}}\lesssim r'p_0'|Q_0|^{\frac{1}{p_0' r '}},
$$
see \cite[p. 128]{gra}. Similarly, we can deduce that
\begin{eqnarray*}
&&\sum_{l}\Big|\int_{P_l}T\big((b-\langle b\rangle_{Q_0})f\chi_{3Q_0\backslash 3P_l}\big)(x)g(x)dx\Big|\\
&&\quad\lesssim \sum_l
|P_l|\langle |g|\rangle_{P_l,\,p_0'}\inf_{y\in P_l}\mathscr{M}_{p_0, T;\,Q_0}\big(b-\langle b\rangle_{Q_0}\big)f(y)\\
&&\quad\lesssim \langle |f|\rangle_{3Q_0}\sum_l|P_l|\langle |g|\rangle_{P_l,\,p_0'}\lesssim \langle |f|\rangle_{3Q_0}\langle |g|\rangle_{Q_0,\,p_0'}|Q_0|.
\end{eqnarray*}
Therefore,
for function $g\in L_{\rm loc }^r(\mathbb{R}^n)$,
\begin{eqnarray}\label{eq4.4}\Big|\int_{\mathbb{R}^n} G_{Q_0}^1(x)g(x)dx\Big|\lesssim r'p_0'\langle |f|\rangle_{3Q_0}\langle |g|\rangle_{Q_0,\,rp_0'}|Q_0|.
\end{eqnarray}
and
\begin{eqnarray}\label{eq4.5}
\Big|\int_{\mathbb{R}^n} G_{Q_0}^2(x)g(x)dx\Big|\lesssim \|f\|_{L\log L,\,3Q_0}\langle |g|\rangle_{Q_0,\,p_0'}|Q_0|.
\end{eqnarray}

We  repeat  argument above with $T(f\chi_{3Q_0})(x)\chi_{Q_0}$ replaced by  $T(\chi_{3P_l})(x)\chi_{P_l}(x)$, and so on.
Let $Q_0^{j_0}=Q_0$, $Q_{0}^{j_1}=P_{j}$,  and for fixed $j_1,\,\dots,\,j_{m-1}$, $\{Q_{0}^{j_1...j_{m-1}j_m}\}_{j_m}$ be the cubes obtained at the $m$-th stage of the decomposition process to the cube $Q_{0}^{j_1...j_{m-1}}$. Set $\mathcal{F}=\{Q_0\}\cup_{m=1}^{\infty}\cup_{j_1,\dots,j_m}\{Q_{0}^{j_1\dots j_m}\}$. Then $\mathcal{F}\subset \mathcal{D}(Q_0)$ is a $\frac{1}{2}$-sparse  family.
We define the functions $H_{1,\,Q_0}$ and $H_{2,\,Q_0}$ by
\begin{eqnarray*}H_{1,Q_0}(x)&=&\sum_{m=1}^{\infty}\sum_{j_1\dots j_{m-1}}(b(x)-\langle b\rangle_{Q_0^{j_1,\dots,j_{m-1}}})\\
&&\qquad\times T(f\chi_{3Q_{0}^{j_1\dots j_{m-1}}})(x)\chi_{Q_{0}^{j_1,\dots,j_{m-1}}\backslash \cup_{j_m}Q_{0}^{j_1,\dots,j_m}}(x)\\
&&+\sum_{m=1}^{\infty}\sum_{j_1\dots j_m}\big(b(x)-\langle b\rangle_{Q_0^{j_1,\dots,j_{m-1}}}\big)\\
&&\quad \times T\big(f\chi_{3Q_{0}^{j_1\dots j_{m-1}}\backslash \cup_{j_{m}}3Q_{0}^{j_1\dots j_{m}}}\big)(x)\chi_{Q_{0}^{j_1\dots j_{m}}}(x),
\end{eqnarray*}
and \begin{eqnarray*}H_{2,Q_0}(x)&=&\sum_{m=1}^{\infty}\sum_{j_1\dots j_{m-1}}T\Big((b(x)-\langle b\rangle_{Q_0^{j_1,\dots,j_{m-1}}})f\chi_{3Q_{0}^{j_1\dots j_{m-1}}}\Big)(x)\\
&&\quad\times\chi_{Q_{0}^{j_1,\dots,j_{m-1}}\backslash \cup_{j_m}Q_{0}^{j_1,\dots,j_m}}(x)\\
&&+\sum_{m=1}^{\infty}\sum_{j_1\dots j_m}T\Big(\big(b(x)-\langle b\rangle_{Q_0^{j_1,\dots,j_{m-1}}}\big)f\chi_{3Q_{0}^{j_1\dots j_m}\backslash \cup_{j_{m+1}}3Q_{0}^{j_1\dots j_{m-1}}}\Big) (x)\\
&&\quad\qquad \times \chi_{Q_{0}^{j_1\dots j_{m-1}}}(x).
\end{eqnarray*}
Then for a. e. $x\in Q_0$,
$$T_{b}(f\chi_{3Q_0})(x)=H_{1,Q_0}(x)+H_{2,Q_0}(x).
$$
Moreover, as in inequalities (\ref{eq4.4})-(\ref{eq4.5}), the process of producing $\{Q_0^{j_1...j_m}\}$ leads to that
$$\Big|\int_{Q_0}g(x)H_{1,Q_0}(x)dx\Big|\lesssim r'p_0'\sum_{Q\in \mathcal{F}}|Q|\langle |f|\rangle_{3Q}\langle|g|\rangle_{Q,\,rp_0'},$$
and
$$\Big|\int_{Q_0}g(x)H_{2,Q_0}(x)dx\Big|\lesssim \sum_{Q\in \mathcal{F}}|Q|\|f\|_{L\log L,\,3Q}\langle|g|\rangle_{Q,\,p_0'}.$$

We can now conclude the proof of Theorem \ref{thm3.1}. In fact, as in \cite{ler4}, we decompose $\mathbb{R}^n$ by cubes $\{R_l\}$, such that ${\rm supp}f\subset 3R_l$ for each $l$, and $R_l$'s have disjoint interiors.
Then for a. e. $x\in\mathbb{R}^n$,
\begin{eqnarray*}T_{b}f(x)=\sum_lH_{1,R_l}f(x)+\sum_lH_{2, R_l}f(x)=:{\rm H}_1f(x)+{\rm H}_2f(x).\end{eqnarray*}
Obviously, ${\rm H}_1$, ${\rm H}_2$ satisfies (\ref{eq4.2x}) and (\ref{eq4.3x}).
Our desired conclusion then follows directly.\qed\end{proof}

\begin{lemma}\label{thm3.2}Let  $\gamma\in \mathbb{N}\cup\{0\}$, $r\in [1,\,\infty)$, and $U$ be an operator. Suppose that for any  bounded function $f$ with compact support, there exists a sparse family of cubes $\mathcal{S}$, such that for any function $g\in L^{r}_{{\rm loc}}(\mathbb{R}^n)$,
\begin{eqnarray}\label{eq3.2}\Big|\int_{\mathbb{R}^n}Uf(x)g(x)dx\Big|\leq \mathcal{A}_{\mathcal{S};\,L(\log L)^{\gamma},\,L^{r}}(f,\,g).\end{eqnarray}
Then for any $w$ with $w^{r}\in A_1(\mathbb{R}^n)$, $\alpha>0$ and bounded function $f$ with compact support,
\begin{eqnarray*}&&w(\{x\in\mathbb{R}^n:\, |Uf(x)|>\alpha\})\lesssim_{n,\,\,\gamma,\,w}\int_{\mathbb{R}^d}\frac{|f(x)|}{\alpha}\log ^{\gamma}\Big({\rm e}+\frac{|f(x)|}{\alpha}\Big)w(x)dx.\end{eqnarray*}
\end{lemma}
\begin{proof} By Theorem 3.2 in \cite{hulai}, we  know that $U$ satisfies the following estimate:
\begin{eqnarray}\label{ine:3.2}
w(\{x\in\mathbb{R}^d:\,|Uf(x)|>1\})&\lesssim &\Big(1+ \Big\{p_1'^{1+\gamma}\big(\frac{p_1'}{r}\big)'\big(t\frac{p_1'/r-1}{p_1'-1}\big)'^{\frac{1}{p_1'}}\Big\}^{p_1}\Big)\nonumber\\
&&\qquad\times\int_{\mathbb{R}^n}|f(y)|\log^{\gamma}({\rm e}+|f(y)|)M_tw(y)dy,
\end{eqnarray}
where $t\in [1,\,\infty)$, $p_1\in (1,\,r')$ such that
$t\frac{p_1'/r-1}{p_1'-1}>1$, and   $M_t$  is defined by
$$M_rf(x)=\big[M(|f|^r)(x)\big]^{1/r}.$$ Let $w^{r}\in A_1(\mathbb{R}^n)$. We choose $\epsilon>0$ such that $w^{r(1+\epsilon)}\in A_1(\mathbb{R}^n)$. Set $t=r(1+\epsilon)$ and $p_1'=2(r-1)\frac{1+\epsilon}{\epsilon}+1$. Then $t\frac{p_1'/r-1}{p_1'-1}=1+\frac{\epsilon}{2}.$ We   obtain from (\ref{ine:3.2}) that
$$w(\{x\in\mathbb{R}^d:\,|Uf(x)|>1\})\lesssim_{n,\gamma,\,w}\int_{\mathbb{R}^n}|f(y)|\log^{\gamma}({\rm e}+|f(y)|)w(y)dy.$$
This, via homogeneity, leads to our desired conclusion.\qed
\end{proof}

{\it Proof of Theorem \ref{thm1.1}}.  By homogeneity, we may assume that $\|\Omega\|_{L^q(S^{n-1})}=1=\|b\|_{{\rm BMO}(\mathbb{R}^n)}$. Let $w^{q'}\in A_1(\mathbb{R}^n)$. We choose $\varepsilon>0$  such that $\varepsilon\in (0,\,\min\{1,\,(q-1)/3\})$ and  $w^{q'(1+\varepsilon)}\in A_1(\mathbb{R}^n)$. On the other hand,
By Lemma \ref{yinli2.2} and Lemma \ref{lem2.4}, we know that for any  $p_0\in (0,\,q/(1+2\varepsilon))$,
$$\|\mathscr{M}_{p_0,\,T_{\Omega}}f\|_{L^1(\mathbb{R}^n)}\lesssim 2^{4\frac{1}{1-p_0\frac{1+2\varepsilon}{q}}}\|f\|_{L^1(\mathbb{R}^n)}.$$
Take $p_0=q/(1+3\varepsilon)$ and $r=\frac{q-(1+3\varepsilon)}{q-1}(1+\varepsilon)$, then $rp_0'=(1+\varepsilon)q'$. Applying Theorem \ref{thm3.1} with such indices $p_0$ and $r$,  we see that   for any bounded function $f$ with compact support, there exists a sparse family of cubes $\mathcal{S}$, such that for any  $g\in L_{{\rm loc}}^{q'(1+\varepsilon)}(\mathbb{R}^n)$,
$$\Big|\int_{\mathbb{R}^n}T_bf(x)g(x)dx\Big|\lesssim p_0'r'2^{4\frac{1+3\varepsilon}{\varepsilon}}\mathcal{A}_{\mathcal{S};\,L\log L,\,L^{q'(1+\varepsilon)}}(f,\,g).
$$
Theorem \ref{thm1.1} now follows from Lemma \ref{thm3.2} immediately.\qed

{\it Proof of Theorem \ref{thm1.2}}.  Again we
assume that $\|\Omega\|_{L^{\infty}(S^{n-1})}=1=\|b\|_{{\rm BMO}(\mathbb{R}^n)}$.
Let $s\in (1,\,\infty)$. Applying (\ref{eq1.6}) and Theorem \ref{thm3.1} (with $p_0=(\sqrt s)'$ and $r=\sqrt s$), we know that for bounded function $f$ with compact support, there exists a $\frac{1}{2}\frac{1}{3^n}$-sparse family of cubes $\mathcal{S}=\{Q\}$,  and functions ${\rm H}_1f$, ${\rm H}_2f$, such that for each   function $g\in L^s_{{\rm loc}}(\mathbb{R}^n)$,
$$\Big|\int_{\mathbb{R}^n}{\rm H}_1f(x)g(x)dx\Big|\lesssim (\sqrt{s})'^2\mathcal{A}_{\mathcal{S};\,L^1,\,L^s}(f,\,g)\lesssim s'^2\mathcal{A}_{\mathcal{S};\,L^1,\,L^s}(f,\,g),$$
$$\Big|\int_{\mathbb{R}^n}{\rm H}_2f(x)g(x)dx\Big|\lesssim (\sqrt{s})'\mathcal{A}_{\mathcal{S};\,L\log L,\,L^{\sqrt{s}}}(f,\,g)\lesssim
s'\mathcal{A}_{\mathcal{S};\,L\log L,\,L^{s}}(f,\,g),$$
and for a. e. $x\in\mathbb{R}^n$,
$$T_{\Omega,b}f(x)={\rm H}_1f(x)+{\rm H}_2f(x).$$
Let $w\in A_1(\mathbb{R}^n)$, $\lambda>0$,  $f$ be a bounded function with compact support. It follows from   Lemma \ref{lem2.5}  that
\begin{eqnarray*}
&&w(\{x\in\mathbb{R}^n:\,|T_{\Omega,\,b}f(x)|>\lambda\})\\&&\quad\le w(\{x\in\mathbb{R}^n:\,|{\rm H}_1f(x)|>\lambda/2\})
+w(\{x\in\mathbb{R}^n:\,|{\rm H}_2f(x)|>\lambda/2\})\\
&&\quad\lesssim [w]_{A_1}[w]^2_{A_{\infty}}\log ({\rm e}+[w]_{A_{\infty}})\int_{\mathbb{R}^n}\frac{|f(x)|}{\lambda}w(x)dx\\
&&\qquad+[w]_{A_1}[w]_{A_{\infty}}\log^2 ({\rm e}+[w]_{A_{\infty}})\int_{\mathbb{R}^n}\frac{|f(x)|}{\lambda}\log\Big({\rm e}+\frac{|f(x)|}{\lambda}\Big)w(x)dx\\
&&\quad\lesssim[w]_{A_1}[w]^2_{A_{\infty}}\log({\rm e}+[w]_{A_{\infty}})\int_{\mathbb{R}^n}\frac{|f(x)|}{\lambda}\log\Big({\rm e}+\frac{|f(x)|}{\lambda}\Big)w(x)dx.
\end{eqnarray*}
This completes the proof of Theorem \ref{thm1.2}.\qed

\medskip

{\bf Acknowledgement}  The research of the second author was supported by the NNSF of
China under grant \#11771399, and the research of the third author was supported by
the NNSF of
China under grant $\#$11871108.


\begin{thebibliography}{99}
\bibitem{abkp} J. Alvarez, R. J. Babgy, D. Kurtz and C. P\'erez, Weighted estimates for commutators of linear operators, Studia Math.
\textbf{104} (1993), 195-209.
\bibitem{cz1}A. P. Calder\'on and A. Zygmund, On the existence of
certain singular integrals, Acta Math. \textbf{88} (1952), 85-139.
\bibitem{cz2} A. P. Calder\'on and A. Zygmund, On  singular
integrals, Amer. J. Math. \textbf{78} (1956), 289-309.

\bibitem{chr2} M. Christ and J.-L. Rubio de Francia, Weak type (1,\,1) bounds for rough operators, II, Invent. Math. \textbf{93} (1988), 225-237.
\bibitem{cpp} D. Chung, M. C. Pereyra, and C. P\'erez, Sharp bounds for general commutators on weighted
Lebesgue spaces, Trans. Amer. Math. Soc. \textbf{364} (2012), 1163-1177.
\bibitem{crw} R. R. Coifman, R. Rochberg and G. Weiss, Factorization theorems for Hardy
spaces in several variables, Ann. of Math., \textbf{103} (1976), 611-635.
\bibitem{duo} J. Duoandikoetxea, Weighted norm inequalities for homogeneous singular integrals, Trans. Amer.
Math.  Soc. \textbf{336} (1993), 869-880.
\bibitem{drf} J. Duoandikoetxea and J. L. Rubio de Francia,
Maximal and singular integrals via Fourier transform
estimates, Invent. Math. \textbf{84} (1986), 541-561.
\bibitem{fp} D. Fan and Y. Pan, Singular integral operators with
rough kernels supported by subvarieties, Amer. J. Math.
\textbf{119} (1997), 799-839.
\bibitem{gra2}  L. Grafakos,  Classical Fourier Analysis, GTM249, 2nd
Edition, Springer, New York, 2008.
\bibitem{gra}  L. Grafakos,  Modern Fourier Analysis, GTM250, 2nd
Edition, Springer, New York, 2008.
\bibitem{gs} L. Grafakos and A. Stefanov,  $L^p$ bounds for singular integrals and maximal singular
integrals with rough kernels, Indiana Univ. Math. J. \textbf{47}
(1998), 455-469.
\bibitem{hu1} G. Hu, $L^p(\mathbb{R}^n)$ boundedness for the commutator of a homogeneous singular integral, Studia Math. \textbf{154} (2003), 13-47.
\bibitem{hulai} G. Hu, X. Lai and Q. Xue, Weighted bounds for the compositions of rough singular integral operators, J. Geom. Anal., to appear, available at arXiv:
    1811.02878.
\bibitem{hlp} T. Hyt\"onen, M. T. Lacey and C. P\'erez, Sharp weighted bounds for the $q$-variation of singular integrals, Bull. Lond. Math. Soc. \textbf{45} (2013), 529-540.
\bibitem{hrt}T. Hyt\"onen, L. Roncal, and O. Tapiola, Quantitative weighted estimates for
rough homogeneous singular integrals,  Israel J. Math. \textbf{218} (2017), 133-164.
\bibitem{john} F. John, Quasi-isometric mappings, In: 1965 Seminari 1962/63 Anal. Alg. Geom. e Topol., 2, Ist. Naz. Alta Mat., Ediz. Cremonese, Rome, pp. 462-473.
\bibitem{ler4} A. K. Lerner, A weak type estimates for rough singular integrals, Rev. Mat. Iberoam. \textbf{35} (2019),  1583-1602.
\bibitem{lpr} K. Li, C. P\'erez, Isreal P. Rivera-Rios and L. Roncal, Weighted norm inequalities for rough singular integral operators, J. Geom. Anal. \textbf{29} (2019), 2526-2564.
\bibitem{prr}C. P\'erez, I. P. Rivera-Rios, L. Roncal, $A_1$ theory of weights for rough
homogeneous singular integrals and commutators, Annali della Scuola normale superiore di Pisa, Classe di scienze. DOI: 10.2422/2036-2145.201608-011.
\bibitem{perez1} C. P\'erez,  Endpoint estimates for commutators of singular integral operators, J. Funct. Anal. \textbf{128}
(1995), 163-185.
\bibitem{rr}M.  Rao and Z. Ren, Theory of Orlicz spaces, Monographs and Textbooks in Pure
and Applied Mathematics, 146, Marcel Dekker Inc., New York, 1991.

\bibitem{rw}F. Ricci and G. Weiss,  A characterization of
$H^1(S^{n-1})$, Proc. Sympos. Pure Math. of Amer. Math. Soc., (S.
Wainger and G. Weiss eds), Vol 35 I(1979), 289-294.
\bibitem{riv} I. P. Rivera-R\'ios, Improved $A_1-A_{\infty}$ and related estimate for commutators of rough singular integrals, Proc. Edinburgh Math. Soc. \textbf{61} (2018), 1069-1086.
\bibitem{se} A. Seeger,  Singular integral operators with rough convolution kernels, J.
Amer. Math. Soc. \textbf{9} (1996), 95-105.
\bibitem{stro}J. O. Str\"omberg, Bounded mean oscillation with Orlicz norms and duality of Hardy spaces,
Indiana Univ. Math. J. \textbf{28} (1979), 511-544.
\bibitem{var} A. M. Vargas, Weighted weak type $(1,\,1)$ bounds for rough operators, J. London Math. Soc.
54 (1996),   297-310.



\bibitem{wil}M. J. Wilson, Weighted inequalities for the dyadic square function without dyadic $A_{\infty}$, Duke Math. J. \textbf{55} (1987), 19-50.
\end{thebibliography}
\end{document}